\PassOptionsToPackage{pdfpagelabels=false}{hyperref}
\documentclass{article}
\usepackage{blindtext}
\usepackage[T2A]{fontenc}
\usepackage[cp1251]{inputenc}
\usepackage[english]{babel}
\usepackage{amsmath}
\usepackage{amsfonts}
\usepackage{amssymb}

\usepackage{amsbib}
\usepackage{amsthm}
\usepackage{mathrsfs}

\usepackage{graphicx}

\numberwithin{equation}{section}

\theoremstyle{plain}

\newtheorem{theorem}{Theorem}
\newtheorem{proposition}{Proposition}[section]
\newtheorem{lemma}{Lemma}[section]

\newtheorem{example}{Example}[section]
\newtheorem{definition}{Definition}[section]
\newtheorem*{Qu}{Question}

\theoremstyle{definition}

\def\Hom{{\hbox{\bf Hom}\;}}

\def\Obj{{\hbox{\bf Obj}}}

\def\C{{\mathbb C}}

\def\G {\Gamma}
\def\A {{\mathcal A}}

\def\* {\ast}

\numberwithin{equation}{section}

\author{А.\,А.~Аrutyunov}
\title{A combinatorial view on derivations in bimodules}

\begin{document}
\maketitle

\begin{abstract}
 This paper is devoted to derivations in bimodules over group rings using previously proposed methods which are related to character spaces over groupoids. The theorem describing the arising spaces of derivations is proved. We consider some examples, in particular the case of $(\sigma, \tau)$-derivations.
\end{abstract}

\section{Introduction}

Derivations in various associative algebras have been actively studied since the middle of the last century. In particular, the following question, known as the Johnson Problem (or ''Derivation Problem''), is widely known:
\begin{Qu}[\cite{Dales}, question 5.6.B]
 Is it true that all derivations in $L_1(G)$ are inner?
\end{Qu}
Here and hereafter $G$ is a finitely generated generally noncommutative group.

A partial answer to this question was given by B. Johnson himself in \cite{Johnson2001}, and the most complete answer was found by V. Losert's in \cite{Losert}. A more detailed description of the history of this problem is given in \cite{AM,Dales}.

In purely algebraic form, consider the group ring $\C[G]$, that is, the space of all linear combinations of the form $\sum_{g\in G} x(g) g$, where $x(\cdot)$ is a finite function --  a function with a finite support. The derivation in this case is a linear operator $d:\C[G]\to \C[G]$ satisfying the Leibniz rule
$$
    d(uv) = d(u)v + u d(v), \quad \forall u,v\in\C[G].
$$
The Johnson's problem is whether there are  derivation other than inner, i.e., having the form
$$
    d_a: x\to [a,x], a\in \C[G].
$$
In this formulation of the problem, the algebra of outer derivations will be nontrivial (see central derivations from \cite{Ar} as well as \cite{AM}).

In the present paper we focus on the study of Banach spaces equipped with a bimodule structure over a group ring $\C[G]$ of the following form. Let us fix a norm $\||\cdot\|$ in $\C[G]$ and let $\A$ be the closure of the group ring by this norm. $\A$ is not an algebra, but it is naturally understood as a bimodule over the ring $\C[G]$.

The ring $\C[G]$ will be defined as space with the supremum norm $\|\cdot\||_s$, i.e. for $\omega = \sum\limits_{g\in G}x(g) g$ give
$$
\|\omega\|_s := \sup\limits_{g\in G} |x(g)|.
$$

The boundedness of operators will be defined as follows. Let us define the norm of the operator $d$ as
\begin{equation}
 \|d\| = \sup\limits_{\omega\neq 0\in\C[G]} \frac{\|d(\omega)\|}{\|\omega\|_s}.
\end{equation}

\begin{definition}
By the derivation over $\C[G]$ with values in bimodule $\A$ we will call a linear bounded operator $d:\C[G]\to\A$ such that
\begin{equation}
 d(uv) = d(u)v + u d(v), \quad \forall u,v\in \C[G].
\end{equation}
The space of such operators we will denote as $Der(\A)$.
\end{definition}

For a wide class of norms (including natural $\ell_p$) the following statement will be proved.
\begin{theorem}
\label{th-norma-qinn}
 If the norm $\|\cdot\|$ is subordinate to the supremum norm, then all derivations in the bimodule $\A$ are quasi-inner.
\end{theorem}

Let us explain about which norms we are talking about.
\begin{definition}
\label{def-norma-podch}
 We will say that the norm $\|\cdot\|$ is subordinated to the norm $\|\cdot\|_s$ iff
 \begin{equation*}
 \|\omega\|<\infty \implies \|\omega\|_s<\infty
 \end{equation*}
\end{definition}

A rigorous definition of quasi-inner derivations is given below (see definition \ref{def-QInn}). Informally speaking, quasi-inner derivations are such operators that can be represented as a formal infinite sum of inner derivations, i.e. the commutator $x\to [a,x]$, for the element $a$ not lying in $\A$ in general.

\section{Preliminaries}

The following definitions and results are based on \cite{AM,Ar-Al}.

Let $\G$ be a groupoid of connected action in which objects coincide with elements of the group $\Obj(\G) = G$, and morphisms are pairs of elements of the group, i.e. $\Hom(\G) = G\times G$. Moreover, the morphism $\phi:=(u,v)$ has a beginning $s(\phi)=v^{-1}u$ and an end $t(\phi) = u v^{-1}$. We will define endomorphisms (i.e. morphisms which source and target coincide) $\phi$ as loops for clarity.

 Then we define a character as a complex-valued function $\chi:\Hom(\G)\to \C$ such that
\begin{equation}
\label{eq-character}
\chi(\psi\circ\phi)= \chi(\psi)+\chi(\phi),
\end{equation}
 for all pairs of linkable morphisms $\phi,\psi$.

The groupoid $\G$ will be represented as an uncoupled union of subgroupoids $\G_{[u]}$, where $[u]$ is a class of conjugate elements.
$$
    [u] = \{tut^{-1}| t\in G\}.
$$
Objects of the subgroupoid $\G_{u}$ coincide with the class $u$.

%The construction above gives is us that each derivation can be presented by character.

\begin{lemma}
\label{lemma-formula}
	For every derivation $d\in Der(\A_s)$ there exists a character $\chi\in X(\G)$ such that the following formula holds
	\begin{equation}
\label{eq-der-formula}
   d(g) = g \left(\sum\limits_{t\in G} \chi(gt,g)t\right), \forall g\in G.
\end{equation}
\end{lemma}

\begin{proof}
	The proof Literally repeats the proof of Theorem 1 of \cite{Ar} except for the property of local finiteness, which was a consequence of the finiteness of nonzero terms in elements of group algebras.
\end{proof}

Let us make it clear that character values are coefficients of the expansion of a linear operator by standard basis of the ring $\C[G]$. Due to the implementation of the Leibniz's rule, the calculation shows that the property \eqref{eq-character} is fulfilled.

The space of quasi-inner derivations was introduced earlier in \cite{Ar-Al}. It will play an important role below.
\begin{definition}
\label{def-QInn}
	We will define the derivation $d\in Der(\A(G))$ as quasi-inner if the character $\chi$ in the formula \eqref{eq-der-formula} is zero on all loops.
\end{definition}

It is easy to see that the inner derivations are quasi-inner. Indeed, let $a\in G$ be a basis element, then the character $\chi^a$ corresponding to the inner derivation $d_a: x\to [x,a]$ has the form (see Section 2.2 in \cite{Ar} for details)

$$
\chi^a (\phi) =
\begin{cases}
 1, & \phi\in \Hom(a,b), b\neq a,
 \\-1 & \phi\in \Hom(b,a), b\neq a,
 \\ 0 & \textit{else}.
\end{cases}
$$

The character $\chi^a$ is zero on all loops. The inner derivation is a sum of (possibly infinite) characters $\chi^a$, and hence is itself trivial on loops.

The role of quasi-inner derivations for group rings is that they form an ideal containing the ideal of inner derivations (Theorem 4.1, \cite{Ar-Al}), with examples of quasi-inner derivations that would not be inner in the group ring (Section 3.3 of \cite{Ar} -- the case of Heisenberg group).

 The question of coincidence spaces of  inner and quasi-inner derivations in a bimodule $\A$ is equivalent to the fact that inner derivations are dense in the space of quasi-inner derivations.

\section{Derivations in bimodules}

Recall that in the group ring $\C[G]$ we have a ''supremum norm''. For $x = \sum\limits_{g\in G} x(g) g\in\C[G]$, where $x(\cdot)$ is a finite function
\begin{equation}
	\|x\|_s:= \sup\limits_{g\in G} |x(g)|.
\end{equation}

The Banach bimodule which is the closure of $\C[G]$ by the supremum norm is denoted by $\A_s(G)$. The bimodule $\A_s(G)$ can be understood as the space of elements of the form $\sum\limits_{g\in G} x(g) g$, where the function $x(g)$ is bounded.

\begin{lemma}
\label{th-suprem}
	All derivations with values in the Banach bimodule $\A_s$ are quasi-inner.
\end{lemma}

\begin{proof}
	Let $\chi$ be the character corresponding in the sense of the lemma \ref{lemma-formula} to a nontrivial derivation $d$. Let us show that if the character $\chi$ takes a nonzero value on some loop, then the operator $d$ is unbounded.

	Consider the loop $\phi\in Hom(g,g)$. As can be seen from the definition, for some $t\in Z(g)$, the notation $\phi = (gt,t)$ is valid.

	Let $\chi (\phi)\neq 0$. Without loss of generality, we may assume that $\chi (\phi) = 1$.

	Note that the element $t$ cannot be of finite order in the group $G$, otherwise we would get that $\chi ((g,e))=ord(t)$, which is impossible since $(g,e)$ is a neutral morphism.

	Given that $t$ is of infinite order, using the property \eqref{eq-character}, we have
	$$
		\chi((gt^n, t^n)) = \chi(\phi^n) = n.
	$$
	with the formula \eqref{eq-der-formula} we get that $\|d(t^n)\|_s\geq n$. So $\|d(t^n)\||_s\to\infty$ at $n\to\infty$. At the same time $\|gt^n\||_s=1$. In this case the operator $d$ -- converts a bounded sequence into an unbounded one, and hence is not bounded itself.
	\end{proof}

\begin{proposition}
 If the character $\chi$ gives the derivation $d\in Der(\A_s)$, then for any two morphisms $\phi,\psi\in\Hom(a,b)$ we have that $\chi(\phi)=\chi(\psi)$.
\end{proposition}
\begin{proof}
	By the lemma \ref{th-suprem} for each derivation $d\in Der(\A_s)$, the corresponding character $\chi$, in the sense of the lemma \ref{lemma-formula}, is zero on loops. For two different objects $a\neq b$, and two morphisms $\phi, \psi\in \Hom(a,b)$ (if they exist, of course) there exists a loop $\zeta\in\Hom(a,a)$ such that $\phi\circ\zeta = \psi$. So, since $\chi(\zeta) = 0$, by the formula \eqref{eq-character} we have that $\chi(\phi)=\chi(\psi)$.
\end{proof}

Let us proceed to the proof of the \ref{th-norma-qinn} theorem, i.e. we show that if norm $\|\cdot\|$ is subject to the supremum norm $\|\cdot\|_s$, then all derivations with values in the bimodule $\A$ are quasi-inner.

\begin{proof}[Proof of theorem \ref{th-norma-qinn}]
Consider a bounded operator $d$ over $\A$. From boundedness we obtain that for some $A$: $\|d(g)\|<A$ for all basis elements $g\in G$.

From the norm subordination we have that $\|d(g)\|_s <CA$. So the character $\chi$ corresponding to the derivation $d$ is trivial on loops by the lemma \ref{th-suprem}, and so $d$ is a quasi-inner derivation.
\end{proof}

\section{Examples and applications}

Consider the space $\ell_p(G)$, i.e., all elements of the form $\omega=\sum_{g\in G}x(g) g$, bounded by the $\ell_p$-norm
\begin{equation}
	\|\omega\|_p:= \sqrt[p]{\sum\limits_{g\in G}|x(g)|^p}.
\end{equation}

It is clear that such norm is subordinate to the supremum if $p\geq 1$:
$$
\|\omega\|_p \leq \sup_{g\in G} |x(g)|.
$$

\begin{example}
	All derivations in $Der(\ell_p(G)), p\geq 1$ are quasi-inner.
\end{example}

In \cite{JR69} was shown that for $p=1$ all derivations are inner. The case $p>1$ is new.

\subsection*{The case of $(\sigma,\tau)$-derivations}

It is clear that in the proof of Theorem \ref{th-norma-qinn} the ability to represent derivations via groupoid characters plays the key role. This construction can be applied to other structures as well, with a corresponding change in the groupoid structure. Thus, in \cite{AAS}, $(\sigma,\tau)$-derivations, i.e., operators satisfying a ''twisted'' Leibniz rule, were investigated. Conditions and applications of $(\sigma,tau)$-derivations can be found in \cite{HLS,ELMS}.

Let us give a slightly more general definition to fit our case.

\begin{definition}
	A linear bounded operator $d:\C[G]\to \A$ such that for some endomorphisms $\sigma,\tau:\C[G]\to\C[G]$
	\begin{equation}
		d(ab) = d(a)\sigma(b) + \tau(a) d(v),\quad a,b\in \C[G],
	\end{equation}
	let's call $(\sigma,\tau)-$derivation.
\end{definition}

As shown in \cite{AAS} for the ''twisted'' groupoid (see \cite{AAS}, Section 3) $(\sigma,\tau)-$derivation can be presented by its characters (\cite{AAS}, Theorem 1). Applying the same reasoning as in the proof of the theorem \ref{th-norma-qinn} to the case of $(\sigma,\tau)-$derivations we obtain the ''twisted'' analog of our theorem.

\begin{example}
	All $(\sigma,\tau)-$derivations with values in bimodule $\A$ generated by a norm subordinate supremum -- will be ($\sigma,\tau$)-quasi-inner.
\end{example}

Here ($\sigma,\tau$)-quasi-inner derivations are defined similarly as ($\sigma,\tau$)-derivations given by characters identically equal to zero on loops.

\subsection*{Central derivations}

The proved theorem and the proposed approach also suggest examples of norms in which outer derivations appear.

Recall (see \cite{Ar}, Section 2.3) that the central derivation $d^{t}_z$ is an operator which is given by a center element of the group $z\in Z(G)$ and a non-trivial homomorphism $t:G\to (\C,+)$ into the additive group of complex numbers on generators $g\in G\subset \C[G]$ by the formula
\begin{equation}
 d^t_z: g\to \tau(g) gz.
\end{equation}
Central derivations form a subalgebra (\cite{Ar}, Theorem 2). It is easy to see that the character $\chi^{t}_z$ corresponding in the sense of the lemma \ref{lemma-formula} is non-trivial on loops (see \cite{Ar}, proposition 5). More precisely, its carrier is a subgroupoid with one object (just a fixed central element). From which we obtain that the central derivation cannot be quasi-inner.

A simple calculation shows that the central derivations are not bounded in the norm supremum, which of course follows from the theorem \ref{th-norma-qinn}. However, it is possible to choose a norm which is not subordinate to the supremum\footnote{The following example is proposed by A. Nayanzin}.

We define the norm $\|\cdot\|^*_{\alpha}$ for $\omega = \sum_{g\in G}x(g) g$ as follows
\begin{equation}
 \|\omega\|^*_{\alpha} := \sum\limits_{g\in G} |x(g)|e^{-\alpha|g|}
\end{equation}

We denote the closure of $\C[G]$ by this norm by $\A^*_{\alpha}$.

\begin{example}
For all sufficiently large $\alpha>0$, central derivations with values in the bimodule $\A^*_{\alpha}$ are bounded.
\end{example}

The proof is achieved by a simple calculation.
%\section{Conclusion}
\bigskip

In the case of a group ring (i.e. the case of finite linear combinations without normalized space structure) the ideals of inner and quasi-inner derivations do not coincide (see \cite{Ar}, the case of Heisenberg group). Moreover, the difference between inner and quasi-inner derivations is determined by the combinatorial properties of the group. As the proof of the theorem \ref{th-norma-qinn} shows in the case of normalized bimodules, this dependence vanishes. So the following assumption seems plausible.
\begin{Qu}
 All derivations in $\ell_p(G)$ are inner.
\end{Qu}

%=================Список литературы====================
%\end{fulltext}


\begin{thebibliography}{10}


\RBibitem{Ar}
\by A.~A.~Arutyunov
\paper Derivation Algebra in Noncommutative Group Algebras
\inbook Differential equations and dynamical systems
\bookinfo Collected papers
\serial Trudy Mat. Inst. Steklova
\yr 2020
\vol 308
\pages 28--41
\publ Steklov Math. Inst. RAS
\publaddr Moscow
\mathnet{http://mi.mathnet.ru/tm4048}
\crossref{https://doi.org/10.4213/tm4048}
\elib{https://elibrary.ru/item.asp?id=43288122}
\transl
\jour Proc. Steklov Inst. Math.
\yr 2020
\vol 308
\pages 22--34
\crossref{https://doi.org/10.1134/S0081543820010022}

\RBibitem{AM}
\by A.~A.~Arutyunov, A.~S.~Mishchenko
\paper A~smooth version of Johnson's problem on derivations of group algebras
\jour Mat. Sb.
\yr 2019
\vol 210
\issue 6
\pages 3--29
\mathnet{http://mi.mathnet.ru/msb9119}
\crossref{https://doi.org/10.4213/sm9119}
\mathscinet{http://www.ams.org/mathscinet-getitem?mr=3954336}
\adsnasa{http://adsabs.harvard.edu/cgi-bin/bib_query?2019SbMat.210..756A}
\elib{https://elibrary.ru/item.asp?id=37652216}
\transl
\jour Sb. Math.
\yr 2019
\vol 210
\issue 6
\pages 756--782
\crossref{https://doi.org/10.1070/SM9119}


\Bibitem{Ar-Al}
\by A.A.~Arutyunov, A.V.~Alekseev
\paper Cohomology of n-categories and derivations in group algebras
\jour Topology and its Applications
\href{https://doi.org/10.1016/j.topol.2019.107002}{https://doi.org/10.1016/j.topol.2019.107002}
\yr 2019


\bibitem{Dales} 
\by Dales, H. G.
\paper Banach algebras and automatic continuity
\publaddr Clarendon Press 
\publ Oxford University Press 
\yr 2000


\Bibitem{Johnson2001}
\by B.E.~Johnson
\paper The derivation problem for group algebras of connected locally compact groups
\jour J. London Math. Soc.
\vol 63:2
\yr 2001
\pages 441-452

\Bibitem{JR69} 
\by Johnson, B. E., Ringrose J. R.
\paper Derivations of operator algebras and discrete group algebras.
\jour Bull. London Math. Soc.
\vol 1
\pages 70-74 
\yr 1969

\Bibitem{AAS}
\by Aleksandr Alekseev, Andronick Arutyunov, Sergei Silvestrov
\paper On $(\sigma,\tau)$-derivations of group algebra as category characters
\jour arxiv
\href{https://arxiv.org/abs/2008.00390}{https://arxiv.org/abs/2008.00390}
\yr 2020

\Bibitem{HLS}
\by  Hartwig, J. T., Larsson, D., Silvestrov, S. D.
\paper Deformations of Lie algebras using $\sigma$ derivations
\jour J. Algebra
\vol 295(2)
\pages 314-361
\yr 2006

\Bibitem{ELMS}
\by Elchinger, O., Lundeng?ard, K., Makhlouf, A., Silvestrov, S. D.
\paper Brackets with $(\tau,\sigma)$-derivations and $(p,q)$-deformations of Witt and Virasoro algebras
\jour Forum Math.
\vol 28(4)
\pages 657-673
\arxiv 1403.6291
\yr 2016


\Bibitem{Losert}
\by V.~Losert
\paper The derivation problem for group algebras
\jour Ann. of Math.
\vol 168:1
\yr 2008
\pages 221-246


\end{thebibliography}
\end{document}